\documentclass[12pt, letterpaper]{amsart}

\usepackage{graphicx}
\usepackage{amssymb}
\usepackage{mathrsfs}
\usepackage{latexsym}
\usepackage{amsmath}
\usepackage[
hypertexnames=false, colorlinks, citecolor=red, linkcolor=red]{hyperref}
\hypersetup{bookmarksdepth=3}

	\normalbaselines						  %

\newcommand{\R}{{\mathbb  R}}

\newcommand{\T}{\mathbb{T}}

\newcommand{\N}{{\mathbb  N}}
\newcommand{\C}{{\mathbb  C}}

\newcommand{\dd}{{\mathrm{d}}}

\newcommand{\bI}{\mathbf{I}}

\newcommand{\fdot}{\,\cdot\,}

\newcommand{\cC}{\mathcal{C}}

\newcommand{\bO}{\mathbf{0}}

\newcommand{\bB}{\mathbf{B}}

\newcommand{\bM}{\mathbf{M}}
\newcommand{\bH}{\mathbf{H}}

\newcommand{\balpha}{{\boldsymbol{\alpha}}}

\newcommand{\bx}{\mathbf{x}}
\newcommand{\by}{\mathbf{y}}

\newcommand{\wt}{\widetilde}

\newcommand{\cH}{\mathcal{H}}

\newcommand{\1}{\mathbf{1}}

\DeclareMathOperator{\Ran}{Ran}

\DeclareMathOperator{\im}{Im}

\DeclareMathOperator{\tr}{tr}

\newcommand{\ci}[1]{_{_{\scriptstyle #1}}}

\newcommand{\ti}[1]{_{\scriptstyle \text{\rm #1}}}

\numberwithin{equation}{section}

\theoremstyle{plain}
\newtheorem{theo}{Theorem}[section]

\newtheorem{lem}[theo]{Lemma}

\theoremstyle{definition}

\theoremstyle{remark}
\newtheorem*{ex*}{Example}
\newtheorem*{exs*}{Examples}
\newtheorem{rem}[theo]{Remark}
\newtheorem*{rem*}{Remark}
\newtheorem*{rems*}{Remarks}

\newcounter{vremennyj}

\title[Aronszajn--Donoghue theorem]{Dimension of the exceptional set in the  Aronszajn--Donoghue theorem for finite rank perturbations }

\author{Constanze~Liaw}
\address{C.~Liaw: Department of Mathematical Sciences, University of Delaware, 311 Ewing Hall, Newark, DE 19716, USA and CASPER, Baylor University, One Bear Place \#97328,      
 Waco, TX  76798, USA}
\email{liaw@udel.edu}

\thanks{Work of C.~Liaw is supported  by the National Science Foundation under the grant  DMS-1802682 }
\author{Sergei~Treil}
\address{S.~Treil: Department of Mathematics, Brown University   
151 Thayer
St./Box 1917,      
 Providence, RI  02912, USA}
 
\email{treil@math.brown.edu}

\thanks{Work of S.~Treil is supported  by the National Science Foundation under the grant  DMS-1856719.}

\author{Alexander Volberg}

\address{A.~Volberg: Department of Mathematics, Michigan State University, East Lansing, MI. 48824, USA}
\email{volberg@msu.edu}
\thanks{Work of A.~Volberg is supported  by the National Science Foundation under the grant  DMS-1900286.}

\begin{document}

\begin{abstract}
The classical Aronszajn--Donoghue theorem states that for a rank one perturbation of a self-adjoint operator (by a cyclic vector) the singular parts of the  spectral measures of the original and perturbed operators are mutually singular. As simple direct sum type examples show, this result does not hold for finite rank perturbations. However, the set of exceptional perturbations is pretty small. 

Namely, for a family of rank $d$ perturbations $A_\balpha := A + \bB \balpha \bB^*$, $\bB:\C^d\to \cH$, with $\Ran\bB$ being cyclic for $A$, parametrized by $d\times d$ Hermitian matrices $\balpha$, the singular parts of the spectral measures of $A$ and $A_\balpha$ are mutually singular for all $\balpha$ except for a small exceptional set $E$. It was shown earlier by the first two authors, see \cite{LTJST}, that $E$ is a subset of measure zero of the space $\bH(d)$ of $d\times d$ Hermitian matrices. 

In this paper we show that the set $E$ has small Hausdorff dimension, $\dim E \le \dim\bH(d)-1 = d^2-1$. 
\end{abstract}

\maketitle

\setcounter{tocdepth}{1}
\tableofcontents

\setcounter{section}{-1}

\section{Introduction}

Consider a family of finite rank (self-adjoint) perturbations of a self-adjoint operator $A$ (possibly unbounded), 
\begin{align}
\label{A_alpha}
A_\balpha := A + \bB \balpha \bB^*, 
\end{align}
parametrized by self-adjoint operators (Hermitian matrices) $\balpha : \C^d \to \C^d$. Here we assume that $\bB:\C^d  \to \cH$ is an injective operator. We do not need to assume that $\bB$ is bounded, it is sufficient to assume that the operator $ (\bI + |A|)^{-1/2} \bB$ is bounded. In this case we are dealing with the so-called ``form bounded'' perturbations; the theory of such perturbations is well-developed, and does not differ much from the case of bounded perturbations (see e.g.~\cite{kurasovbook}). 

Isolating the interesting from the perturbation theory point of view case we always assume that $\Ran \bB $ is cyclic for $A$. In the case of rank-one perturbations ($d=1$) the classical Aronszajn--Donoghue Theorem states that the singular parts of the spectral measures of $A$ and $A_\balpha$ are always mutually singular. 

As simple direct sum examples show, this is not the case for $d>1$. So, the singular parts of the spectral measures of the original and perturbations, and the singular parts of the \emph{scalar} spectral measures of $A$ and $A_\balpha$   are not always mutually singular. However, it was proved in \cite{LTJST} that they  are mutually singular for almost all perturbations. 

Moreover, it was proved in \cite{LTJST} that if $\balpha_1 >\bO$ (i.e.~positive definite) and $\balpha(t)=\balpha_0 + t\balpha_1 $,  then given a singular measure $\nu$  the spectral measures $\mu^t$ (equivalently their singular parts $\mu^t\ti s$) of the operators $A_{\balpha(t)}$ are mutually singular with $\nu$ for all $t\in\R$ except, maybe, countably many. 

This leads one to suspect that in fact one can say more about the exceptional set, i.e. the set of all Hermitian $d\times d$ matrices $\balpha$ for which the singular parts of the scalar spectral measures%
\footnote{By a scalar spectral measure we always mean a scalar spectral measure of maximal dimension.}
of $A$ and of  $A_\balpha$ are \emph{not mutually singular}.  It looks like a reasonable conjecture  that the exceptional set is not just a set of measure $0$, but it in fact  has dimension strictly less than the full dimension $d^2$ of the set $\bH(d)$ of all $d\times d $ Hermitian matrices. (It is not hard to see that the Hausdorff dimension of the set  $\bH(d)$ of all $d\times d $ Hermitian matrices is exactly $d^2$.) 

This turns out to be the case; the main result of this note is the following theorem:

\begin{theo}
\label{t:main}
Let operators $A_\balpha$ be given by \eqref{A_alpha}, and let $\Ran \bB$ be cyclic for $A$. 
Given a singular measure $\nu$ the scalar spectral measures $\mu^\balpha$ of the operators $A_\balpha$ 
are mutually singular with $\nu$ for all $\balpha\in\bH(d)\setminus E$, where the exceptional set $E$ has  Hausdorff dimension at most $\dim \bH(d)-1 = d^2 - 1$. 
\end{theo}


\section{An application of the Marstrand--Mattila theorem}
\label{s:Marstrand-Mattila}
The tool to show that the dimension of the exceptional set is at most $\dim \bH(d) - 1 =d^2 -1$ is already available. 
Namely, the following result was proved in \cite[Lemma 6.4]{Mattila1975}. Below, $\cH^s$ denotes the $s$-dimensional Hausdorff measure, and $G(m,n)$ denotes the set of all $m$-dimensional subspaces of $\R^n$. 

\begin{lem}
	\label{l:Mattila--Marstrand} Let $E$ be an $\cH^s$ measurable subset of $\R^n$ with $0<\cH^s(E)<\infty$. Then 
	\[
	\dim( E\cap (V+x)) \ge s+m-n
	\]
	for almost all $(x,V)\in E\times G(m,n)$. 
\end{lem}

In this lemma \emph{$\cH^s$ measurable} means \emph{Carath\'{e}odory measurable} with respect to the outer measure $\cH^s$. 

This result was proved by J.~M.~Marstrand in \cite{Marstrand1954} for $n=2$ and by P.~Mattila \cite{Mattila1975} for general  $n\in\N$. 

A formal application of this result would immediately give us the desired estimate on the dimension (see the reasoning at the end of this section). However, we do not know anything about the exceptional set $E$ (see the definition below in Section \ref{s:prelims}); we do not know whether it is $\cH^s$ measurable for $s>d^2-1$. But what is more important, we cannot say that $\cH^s(E)<\infty$ for $s>d^2-1$.

However, as it was discussed in \cite{Mattila1975}, if one assumes that $E$ is an analytic (a.k.a.~a Suslin) set, one can reduce the assumption to $\cH^s(E)>0$. The reason for this reduction is that, by the theorem of R.~O.~Davies \cite{Davies1952}, given an analytic set $E\subset \R^n$ with $\cH^s(E)>0$ one can find a compact $K\subset E$ with $0<\cH^s(E)<\infty$. 

So, while it was not stated explicitly, the following statement  was proved in \cite{Mattila1975}. 

\begin{lem}
	\label{l:Mattila--Marstrand 02} Let $E$ be an analytic (Suslin) subset of $\R^n$ such than $\cH^s(E)>0$. Then 
	\[
	\dim( E\cap (V+x)) \ge s+m-n
	\]
	for almost all $(x,V)\in E\times G(m,n)$.
\end{lem}

We will not be giving the definition of analytic (Suslin) sets; for our purposes it is sufficient to know that every Borel set in $\R^n$ is analytic, cf. \cite[Sec.~6.6]{Bogachev2007vol2}. 

We do not know if the exceptional set  $E$ is analytic. However we will be able to prove the following statement, which essentially is the main result of this paper. 
\begin{theo}
	\label{t:Borel}
	There exists a Borel set $\wt E\subset \bH(d)$ such that $E\subset \wt E$ and such that the intersection of $\wt E$ with any line with the direction from the open cone of positive definite Hermitian matrices is at most countable.  
\end{theo}

\begin{proof}[Proof of Theorem \ref{t:main} using Theorem \ref{t:Borel}]
For a moment let us  assume that $\dim \wt E> \dim \bH(d)-1=d^2 -1$. Then $\cH_s(\wt E)>0$ for some $s>d^2-1$.  Applying  Lemma \ref{l:Mattila--Marstrand 02} with $n=\dim \bH(d)=d^2$ and $m=1$ we see that for almost all lines in $\bH(d)$ of the form $\balpha_0 + t \balpha$, $\balpha_0\in \wt E$, $\balpha\in \bH(d)$, their intersection with the extended exceptional set $\wt E$ should have positive  Hausdorff dimension (at least $s+1-d^2>0$). But as we just discussed above, for all such lines $L$ with the directions $\balpha$ in the open cone of positive definite matrices $\balpha$ (i.e.~for a set of non-zero measure) we have at most countable intersection, so the dimension of the intersection is $0$. This gives a contradiction, and therefore $\dim \wt E\le \dim \bH(d)-1 = d^2 -1$. 
\end{proof}

\section{Preliminaries: spectral measures and the exceptional set}\label{s:prelims}

For the operator $A$ with cyclic set $\Ran\bB$ define its matrix-valued spectral measure $\bM$ with values in $\bH(d)$ as the unique measure satisfying
\begin{align*}
\bB^* (A - z \bI)^{-1} \bB = \int_\R (s - z)^{-1} \dd\bM(s) = :\cC \bM(z);
\end{align*}
the spectral measures $\bM_\balpha$, $\balpha\in\bH(d)$ are defined the same way with $A$ replaced by $A_\balpha$. 

It follows from the standard resolvent identities that the Cauchy Transforms $\cC \bM$ and $\cC \bM_\balpha$ are related by the following well known formula
\begin{equation}
\label{M - M_alpha}
\cC \bM_\balpha = \cC\bM (\bI + \balpha\, (\cC\bM))^{-1} = (\bI +  (\cC\bM) \balpha)^{-1}\cC\bM.
\end{equation}
For a proof of these relations, see e.g.~\cite[Lemma 3.1]{LTJST}.

Define the scalar measure 
$\mu^\balpha:= \tr \bM_\balpha$. Clearly $\bM_\balpha$ is absolutely continuous with respect to $\mu^\balpha$, 
\begin{align*}
\dd \bM_\balpha = W_\balpha \dd\mu^\balpha, \qquad \| W(s) \|\le 1 \quad \mu^\balpha\text{-a.e.}
\end{align*}
It is not hard to see that $\mu^\balpha$ is a scalar spectral measure of the operator $A_\balpha$ (recall that by a scalar spectral measure we always mean a scalar spectral measure of maximal type). 
Recall, that a scalar spectral measure is not unique, it is defined up to a multiplication by a non-vanishing weight. So the \emph{spectral type} $[\mu^\balpha]$ of $\mu^\balpha$, i.e.~the equivalence class of all mutually absolutely continuous with $\mu^\balpha$ measures, gives us all possible scalar spectral measures of $A_\balpha$. For our purposes it does not matter which representative we choose, and $\mu^\balpha$ is a convenient choice.

For a fixed finite singular measure $\nu$ on $\R$ define the exceptional set $E=E(\nu)$ to be the set of all $\balpha\in\bH(d)$ for which the measures $\mu^\balpha$  and $\nu$ are not mutually singular. Note that in the definition of $E$ we can replace $\mu^\balpha$ by its singular part, and the resulting set will be exactly the same. 

 If $\nu$ is the singular part of the spectral measure $\mu^{\balpha_0}$, then the exceptional set $E$ is exactly the set of all $\balpha\in\bH(d)$ such that the singular parts of $\mu^{\balpha_0}$ and $\mu^\balpha$ are \emph{not mutually singular} for some $\balpha$.

 Note, that the set $E$ is explicitly defined, not just up to a set of measure zero. In other words, for each $\balpha\in\bH(d)$ one can always say  if $\balpha\in E$ or not.

However we do not know  whether the set $E$ is Borel, or even a Suslin (analytic) set. In particular, we cannot directly approach the set $E$ using measure theoretic tools. We bypass this problem by constructing a bigger set $\wt E$, which is Borel, see the above Theorem \ref{t:Borel}.

Recall that for a (say finite) Borel measure $\mu$ on $\R$ its Poisson extension to the upper half-plane $\C_+$ (which we, slightly abusing notation, denote as $\mu(z)$) is given as 
\begin{align*}
\mu(z) = \pi^{-1}\im \cC \mu(z), \qquad z \in \C_+, 
\end{align*}
where $\cC\mu$ is the Cauchy transform of the measure $\mu$, 
\[
\cC \mu(z)= \int_\R (s-z)^{-1} \dd\mu(s), \qquad z\in \C_+. 
\]
Similarly, for a (say again finite) matrix measure $\bM$ its Poisson  extension $\bM(z)$ to the point $z\in\C_+$ is given by 
\[
\bM(z) = \pi^{-1} \im \cC\bM(z), \qquad \im \cC \bM(z) = (2i)^{-1} (\cC\bM(z) - \cC\bM(z)^* ). 
\]

This formula, together with \eqref{M - M_alpha} implies that the functions $(\balpha, z)\mapsto \bM_\balpha(z)$, $(\balpha, z)\mapsto \mu^\balpha(z) =\tr \bM_\balpha(z)$ is a continuous function of the arguments $\balpha\in\bH(d)$, $z\in\C_+$. 

\section{Some measure theory}
\label{s:measure}

It is well known (see e.g.~\cite[Part (ii) of Theorem 3.4]{LTJST}) that for a finite non-negative measure $\mu$ on $\R$ its singular part $\mu\ti s$ is \emph{carried} by the set $S_\mu$ of all $x\in\R$ for which 
\[
\lim_{z\to x\sphericalangle} \mu(z) = + \infty;
\]
here by limit we understand the non-tangential limit, and  $\mu(z)$ is the Poisson extension of the measure  $\mu$ at the point $z\in\C_+$. The term \emph{carried} here means that $\mu\ti s(\R\setminus S_\mu) =0$. 

Note, that if we pick a \emph{reasonable} sequence  $y_n \downarrow 0$, for example $y_n = 2^{-n}$ (or $y_n= 1/n$), then it follows easily from Harnack's inequality that for any aperture of the approach region 
\[
\lim_{z\to x\sphericalangle} \mu(z) = + \infty \quad \text{if and only if } \quad \lim_{n\to \infty} \mu(x+i y_n) = + \infty;
\]
this equivalence holds for \emph{all} $x\in \R$. In particular, this means that the set $S_\mu$ is always a Borel set. 

Let us fix such a ``reasonable'' sequence $y_n$. 
Define the set $F^1\subset \bH(d)\times \R$, consisting of all pairs $(\balpha, x)$ such that 
\[
\lim_{n\to \infty} \mu^\balpha(x + i y_n)=+\infty. 
\]		

As we discussed  at the end of Section \ref{s:prelims},  $(\balpha, z)\mapsto\mu^\balpha(z)$ is a continuous (and so Borel measurable) function of the variables $\balpha\in\bH(d)$, $z\in\C_+$. Therefore,  the functions $(\balpha, x)\mapsto \mu^\balpha(x+iy_n)$ are  Borel measurable functions of the variables $(\balpha, x)\in \bH(d)\times \R$,  so the set $F^1$ is a Borel subset of $\bH(d)\times \R$.

It is well known that if $\dd\nu = w\dd\mu$, then 
\[
\lim_{z\to x\sphericalangle} \nu(z)/\mu(z) = w(x) \qquad \mu\text{-a.e.}
\]
This implies that for the ``reasonable'' sequence of  $y_n$ we picked above, we have
\[
\lim_{n\to\infty} \nu(x+i y_n)/\mu(x+i y_n) = w(x) \qquad \text{for }\mu\text{-a.e. }x\in\R.
\]

As we discussed above, $\dd\bM_\balpha = W_\balpha\dd \mu^\balpha$, 
$\mu^\balpha$-a.e.,     so for any fixed $\balpha\in\bH(d)$   
\begin{align}
\label{density=limit}
\lim_{n\to\infty} \bM_\balpha(x+i y_n)/\mu^\balpha(x+ i y_n) = W_\balpha(x) \qquad \mu^\balpha\text{-a.e.}
\end{align}
Define the set $F^2\subset \bH(d)\times \R$ to be the set of all pairs $(\balpha, x)$ such that the limit  $\lim_{n\to \infty} \bM_\balpha(x +i y_n)/\mu^\balpha(x +i y_n) $ exists and is non-zero. Again, this set is clearly Borel. 

Moreover, if for a set $F\subset \bH(d)\times\T$ we denote by $F_\balpha$ its section, 
\[
F_\balpha:=\{x\in\R: (\balpha, x )\in F\},
\] 
then for any $\balpha \in\bH(d)$  we have 
\begin{align}
\label{IgnoreLim0}
\dd\bM_\balpha = \1_{F^2_\balpha} W_\balpha \dd\mu^\balpha
\end{align}
(i.e.~defining the density $W_\balpha$ we can ignore the set where the limit \eqref{density=limit} does not exists or equals $\bO$). 

Noticing that for any Borel $E\subset\R$ the measure $\mu^\balpha(E)>0$ if and only if $\bM_\balpha(E)\ne \bO$, we get the following statement. 

Let $F=F^1\cap F^2$, so according to our notation $F_\balpha=F^1_\balpha \cap F^2_\balpha$.

\begin{lem}
\label{l:SingCarrier}
For any $\balpha\in\bH(d)$ the singular part of $\mu^\balpha$ is carried by $F_\balpha$, meaning that $\mu^\balpha\ti s (\R\setminus F_\balpha)=0$. 
This implies, in particular, that given a singular measure $\nu$ the measures $\nu$ and $\mu^\balpha$ (equivalently $\nu$ and the singular part of $\mu^\balpha$) are not mutually singular only if $\nu(F_\balpha)>0$. 
\end{lem} 

\begin{proof}
As we discussed above, the singular part of $\mu^\balpha$ is supported on the set $F^1_\balpha$. Formula \eqref{IgnoreLim0} implies that $\bM_\balpha (F^1_\balpha\setminus F^2_\balpha)=\bO$. Since $\mu^\balpha=\tr \bM_\balpha$, we conclude that $\mu^\balpha( F^1_\balpha\setminus F^2_\balpha ) =0$. So the singular part of $\mu^\balpha $ is indeed supported on $F_\balpha$. 

The second statement follows trivially. 
\end{proof}

\begin{lem}
\label{l:tilde E}
Let $\wt E = \wt E(\nu)$ be the set of all $\balpha\in \bH(d)$ such that $\nu(F_\balpha)>0$. Then $\wt E$ contains the exceptional set $E$ and $\wt E$ is Borel.
\end{lem}

\begin{proof}
The containment $E\subset \wt E$ follows from Lemma \ref{l:SingCarrier}. The Borel measurability of $\wt E$ is an immediate corollary of the Tonelli theorem.
\end{proof}

\section{Directional support of the singular part}
\label{s:DirectionalSupport}

\begin{lem}
\label{l:alpha orthog}
Suppose that for  $\balpha\in \bH(d)$ and for $x\in\R$ 
\begin{align}
\label{lim mu = infty}
\lim_{n\to \infty} \mu(x+ i y_n )  =+\infty, \qquad
\lim_{n\to \infty} \mu^\balpha(x+ i y_n)  =+\infty, 
\end{align}
and that the limits
\begin{align}
\label{lim M/mu}
&\lim_{n\to \infty} \bM (x + i y_n)/\mu(x + i y_n)  =: W(x), \\
\notag 
& \lim_{n\to \infty} \bM_\balpha(x + i y_n)/\mu^\balpha(x + i y_n) =: W_\balpha(x)
\end{align}
exist and are non-zero.  

Then 
\begin{align*}
\Ran W(x) \perp \balpha \Ran W_\balpha (x)\qquad \text{or, equivalently,} \qquad 
\balpha \Ran W(x) \perp \Ran W_\balpha (x) .
\end{align*}
\end{lem}

\begin{rem*}
We do not need to assume that the limits are non-zero: if one of the limits is zero, the statement is trivial. 
\end{rem*}

\begin{rem}
The above Lemma \ref{l:alpha orthog} looks very much like Theorem 6.2 from \cite{LTJST}, and the proof below is essentially the proof from \cite{LTJST}. But the important difference is that in \cite{LTJST} the orthogonality condition was satisfied $\mu\ti s + \mu^\balpha\ti s$ a.e., but in Lemma \ref{l:alpha orthog} we need it to hold for \emph{all} $x$ satisfying \eqref{lim mu = infty} and \eqref{lim M/mu}. Since the devil is often in details, we present a complete proof below. 
\end{rem}

\begin{proof}[Proof of Lemma \ref{l:alpha orthog}]
By Theorem 6.7 from \cite{LTJST} the matrix measures $\bM$ and $\wt\bM_\balpha:=\balpha \bM_\balpha \balpha$ satisfy the joint \emph{two weight matrix $A_2$ condition,} i.e. 
\[
\left\| \bM(z)^{1/2}  \wt  \bM_\balpha(z)^{1/2} \right\| \le C<\infty \qquad \forall z \in \C_+.
\]
We can rewrite this inequality as 
\begin{align}
\label{mA_2-01}
(\mu(z) \mu^\balpha(z))^{1/2}\left\| \left( \vphantom{\wt \bM _\balpha }\bM (z)/\mu(z)\right)^{1/2}  \left( \wt  \bM _\balpha(z) /\mu^\balpha(z)\right)^{1/2} \right\| \le C<\infty \qquad \forall z \in \C_+.
\end{align}
Let us substitute $z=z_n = x+iy_n$ from the statement of the lemma into \eqref{mA_2-01} and take the limit as $n\to\infty$. Taking \eqref{lim mu = infty} into account we can conclude from \eqref{mA_2-01} that 
\begin{align}
\label{lim norm = 0}
\lim_{n\to\infty} \left\| \Big( \bM (z_n)/\mu(z_n)\Big)^{1/2}  \Big( \wt  \bM _\balpha(z_n) /\mu^\balpha(z_n)\Big)^{1/2} \right\| = 0 .
\end{align}
It follows from the identities \eqref{lim M/mu} that 
\begin{align*}
\lim_{n\to\infty} \bM(z_n)/\mu(z_n) = W(x), \qquad
\lim_{n\to\infty} \wt\bM_\balpha(z_n)/\mu^\balpha(z_n) = \balpha W_\balpha(x) \balpha , 
\end{align*}
so the limit in \eqref{lim norm = 0} is exactly
\begin{align*}
\left\| \big( W(x) \big)^{1/2} \big(\balpha W_\balpha (x) \balpha \big)^{1/2}     \right\| = 0. 
\end{align*}
But the last identity could happen only if the ranges of (self-adjoint) matrices $W(x)$ and $\balpha W_\balpha(x) \balpha$ are orthogonal. 
\end{proof}

\section{Countable intersections with extended exceptional set}
\label{s:count intersections}

In the theorem below $F_\balpha$ is  defined as in Section \ref{s:measure}.

\begin{theo}
	\label{t:countable_intersections}
Let $\balpha, \balpha_0\in\bH(d)$ with $\balpha >\bO$, and let $\balpha(t) = \balpha_0 + t \balpha$
Then, given a finite singular Borel measure $\nu$ on $\R$, for all $t$ except maybe countably many
	\[
	\nu( F_{\balpha(t)}) =0 . 
	\] 
\end{theo}

We need the following  lemma, the trivial proof of which we chose to omit.

\begin{lem}
	\label{l:A-ort-distance}
Let $A$ be a positive definite matrix. 	There exists a constant $c=c(A)$ depending on  $A$ such that the condition $(A\bx, \by) =0$ implies
\[
\| \bx-\by\|^2 \ge c(A) ( \| \bx\|^2 + \| \by\|^2 ). 
\]

\end{lem}

\begin{proof}[Proof of Theorem \ref{t:countable_intersections} using Lemma \ref{l:A-ort-distance}.]
Pick
$t,t'\in\R$ such that $\nu(F_{\balpha(t)})\ne0$ and $\nu(F_{\balpha(t')})\ne0$. 
Recall that for all $(\balpha, x) \in F$ (where $F$ is as defined in Section \ref{s:measure}) the limit 
\[
\lim_{n\to \infty} \bM_\balpha (x + i y_n)/\mu^\balpha(x + i y_n) =: W_\balpha(x) 
\]
exists and is non-zero. 
	
	We had shown in Section \ref{s:measure} that the set $F$ is Borel measurable and that the function  $(\balpha, x)  \mapsto W_\balpha(x)$ is a measurable function on $F$ (as a limit of a sequence of continuous functions). Extending it by $\bO$ outside of $F$ we will get a measurable function defined on the whole $\bH(d)$. 
	
	It is then an easy exercise to show that  we can find a vector-valued Borel measurable function $(\balpha,x)\mapsto \Phi(\balpha, x)\in \Ran W_\balpha(x) \subset \C^d$ such that $\Phi(\balpha, x)\ne\bO$ if and only if $(\balpha, x) \in F$. Multiplying this function by an appropriate measurable function depending on $\balpha$ only we can assume that without loss of generality 
	\[
	\int_\R \|\Phi(\balpha, x)\|\ci{\C^d}^2 \dd\nu(x) = 1 \qquad 
	\]
	whenever $\nu(F_\balpha)\ne0$; if $\nu(F_\balpha) = 0$ the integral is trivially $0$. 
	
	Define $f(x):=\Phi(\balpha(t),x)$, $g(x): = \Phi(\balpha(t'),x)$. Note  that 
	\[
	\|f\|\ci{L^2(\nu)} = \|g\|\ci{L^2(\nu)} =1.
	\]
	
	Since $\balpha(t') = \balpha(t) + (t'-t)\balpha$,
Lemma \ref{l:alpha orthog} (applied to $W_{\balpha(t)}$ instead of $W$ and $W_{\balpha(t')}$ instead of $W_\balpha$) implies that $\balpha \Ran W_{\balpha(t)} \perp \Ran W_{\balpha(t')} $, so
\[
(\balpha f(x) , g(x)  )\ci{\C^d} = 0 \qquad \forall x\in \R;
\]  
if both points $(\balpha(t), x)$ and $(\balpha(t'), x)$ are in $F$, this follows from Lemma \ref{l:alpha orthog}; if not, this is trivial, because one of the vectors $f(x)$, $g(x)$ is zero. 

Applying Lemma \ref{l:A-ort-distance} 
we get that for all $x\in\R$
\[
\|f(x) - g(x)\|\ci{\C^d}^2 \ge c \cdot \left( \|f(x) \|\ci{\C^d}^2  + \|g(x)\|\ci{\C^d}^2  \right) .
\]
Integrating this inequality we see that 
\[
\|f-g\|\ci{L^2(\nu)}^2 \ge c \cdot \left( \|f(x) \|\ci{L^2(\nu)}^2  + \|g(x)\|\ci{L^2(\nu)}^2  \right)  = c .
\]

So, for all $t\in \R$ such that $\nu(F_{\balpha(t)})\ne0$ we constructed  unit vectors $f_t = \Phi(\balpha(t), \fdot) \in L^2(\nu)= L^2(\nu;\C^d) $ such that the distance between any two such vectors is at least some fixed  $c>0$. Since the space $L^2(\nu;\C^d) $ is separable, there can be at most countably many such $t\in\R$.  
\end{proof}

\providecommand{\bysame}{\leavevmode\hbox to3em{\hrulefill}\thinspace}
\providecommand{\MR}{\relax\ifhmode\unskip\space\fi MR }
\providecommand{\MRhref}[2]{%
  \href{http://www.ams.org/mathscinet-getitem?mr=#1}{#2}
}
\providecommand{\href}[2]{#2}

\end{document}